\documentclass[11pt, a4paper, leqno]{article}
\usepackage[margin=2.1cm]{geometry}
\usepackage[utf8]{inputenc} %---latin2, cp1250
\usepackage{lmodern}
\usepackage[T1]{fontenc}
\usepackage[english]{babel}

\usepackage[runin]{abstract}
\usepackage{titling}

\usepackage{amsmath}
\usepackage{amsthm}
\usepackage{amsfonts}
\usepackage{amssymb}
\usepackage{rsfso}%---for math calligraphic letters
\usepackage{mathtools}
\usepackage{clrscode}
\usepackage{enumitem}
\usepackage{mdwlist}

\usepackage[nocompress]{cite}%---for extra cite features

\abslabeldelim{.}
\newcommand{\keywordsname}{Key words}
\makeatletter
\newcommand{\keywords}[1]{%
\def\thekeywords{#1}%
\begin{@bstr@ctlist}
\hspace*{\abstitleskip}{\abstractnamefont\keywordsname\@bslabeldelim}\abstracttextfont\
#1%
\par\end{@bstr@ctlist}
}
\makeatother
\newcommand{\subjclassname}{Mathematics subject classification}
\makeatletter
\newcommand{\subjclass}[2][2010]{%
\begin{@bstr@ctlist}
\hspace*{\abstitleskip}{\abstractnamefont\subjclassname\ (#1)\@bslabeldelim}\abstracttextfont\
#2%
\par\end{@bstr@ctlist}
}
\makeatother
\makeatletter
\def\and{%				%begin{tabular}
	\end{tabular}%
	and%
	\begin{tabular}[t]{c}}%
\makeatother
\makeatletter
\def\thanks#1{%\footnotemark
\protected@xdef\@thanks{\@thanks
\protect\footnotetext[\the\c@footnote]{#1}}%
}
\makeatother
\makeatletter
\let\addresses\@empty      %\let\thankses\@empty
\newcommand{\address}[2][]{\g@addto@macro\addresses{\address{#1}{#2}}}
\newcommand{\curraddr}[2][]{\g@addto@macro\addresses{\curraddr{#1}{#2}}}
\newcommand{\email}[2][]{\g@addto@macro\addresses{\email{#1}{#2}}}
\newcommand{\urladdr}[2][]{\g@addto@macro\addresses{\urladdr{#1}{#2}}}
\def\enddoc@text{%\ifx\@empty\@translators \else\@settranslators\fi
  \ifx\@empty\addresses \else\@setaddresses\fi}
\AtEndDocument{\enddoc@text}
\def\emailaddrname{E-mail address}
\def\@setaddresses{\par
  \nobreak \begingroup
  \interlinepenalty\@M
  \def\address##1##2{\begingroup%
    \par\addvspace\bigskipamount%\indent
    \@ifnotempty{##1}{(\ignorespaces##1\unskip) }%
    {\noindent\ignorespaces##2}\par\endgroup}%
  \def\email##1##2{\begingroup
    \@ifnotempty{##2}{\nobreak\noindent\emailaddrname
      \@ifnotempty{##1}{, \ignorespaces##1\unskip}\/:\space
      \ttfamily##2\par}\endgroup}%
  \addresses
  \endgroup
}

\makeatother
\makeatletter
\def\cstar#1{\expandafter\@cstar\csname c@#1\endcsname}
\def\@cstar#1{\ifcase#1\or $\ast$\or $\ast\ast$\or $\ast\ast\ast$\fi}
\AddEnumerateCounter{\cstar}{\@cstar}{$\ast\ast\ast$}
\makeatother

\newlist{conditions}{enumerate}{1}
\newlist{iconditions}{enumerate}{1}
\newlist{nconditions}{enumerate}{1}
\setlist[conditions]{label=\normalfont(\alph*),ref=\normalfont{(\alph*)}}
\setlist[iconditions]{label=\normalfont(\roman*),ref=\normalfont{(\roman*)}}
\setlist[nconditions]{label=\normalfont(\arabic*),ref=\normalfont{(\arabic*)}}

\mathchardef\mhyphen="2D
\newcommand{\R}{\mathbb{R}}

\newcommand{\C}{\mathcal{C}}
\newcommand{\FC}{\mathcal{F}}
\newcommand{\GC}{\mathcal{G}}
\newcommand{\NC}{\mathcal{N}}
\newcommand{\OC}{\mathcal{O}}
\newcommand{\RC}{\mathcal{R}}

\newcommand{\mf}{\mathfrak{m}}
\newcommand{\pf}{\mathfrak{p}}

\newcommand{\Cinfty}{\C^{\infty}}
\newcommand{\Wns}{W^{\mathrm{ns}}}
\newcommand{\Vsing}{V^{\mathrm{sing}}}
\newcommand{\Wsing}{W^{\mathrm{sing}}}

\newcommand{\Max}{\func{Max}}
\newcommand{\Rad}{\func{Rad}}
\newcommand{\Reg}{\func{Reg}}
\newcommand{\Spec}{\func{Spec}}

\newtheorem{theorem}{Theorem}[section]
\newtheorem{corollary}[theorem]{Corollary}
\newtheorem{proposition}[theorem]{Proposition}
\theoremstyle{definition}
\newtheorem{definition}[theorem]{Definition}
\newtheorem{example}[theorem]{Example}
\newtheorem{notation}[theorem]{Notation}
\newtheorem{problem}[theorem]{Problem}
\newtheorem{remark}[theorem]{Remark}

\usepackage[pdftex]{hyperref}

\title{\texorpdfstring{\bf}{}Nash regulous functions}
\date{}
\author{Wojciech Kucharz\thanks{\small The author was partially supported by
the National Science Centre (Poland) under grant number
2014/15/B/ST1/00046.}}

\address{Institute of Mathematics\\Faculty of Mathematics and Computer
Science\\Jagiellonian University\\\L{}ojasiewicza 6\\30-348
Krak\'ow\\Poland}
\email{Wojciech.Kucharz@im.uj.edu.pl}

\postauthor{\end{tabular}\par\end{center}\vskip 0.5em
\begin{center}\small
To Kamil Rusek on his 70th birthday
\end{center}%
}
\hypersetup{pdfauthor={\theauthor}, pdftitle={\thetitle}}

\begin{document}
\maketitle
\thispagestyle{empty}

\begin{abstract}
A real-valued function on $\R^n$ is \emph{$k$-regulous}, where $k$ is a
nonnegative integer, if it is of class $\C^k$ and can be represented as
a quotient of two polynomial functions on $\R^n$. Several
interesting results involving such functions have been obtained
recently. Some of them (Nullstellensatz, Cartan's theorems A~and~B,
etc.) can be carried over to a new setting of Nash $k$-regulous
functions, introduced in this paper. Here a function on a Nash manifold
$X$ is called \emph{Nash $k$-regulous} if it is of class $\C^k$ and can
be represented as a quotient of two Nash functions on $X$.
\end{abstract}

\hypersetup{pdfkeywords={\thekeywords}}

%Section 1
\section{Introduction}\label{sec-1}

Throughout this paper by a function on a set $S$ we always mean a
real-valued function. Given a collection $F$ of functions on $S$, we set
\begin{equation*}
Z(F) \coloneqq \{ x \in S \colon f(x) = 0 \quad \textrm{for all}\ f \in F \}
\end{equation*}
and write $Z(f_1, \ldots, f_r)$ for $Z(F)$ if $F=\{ f_1, \ldots, f_r\}$.

A function $f$ on $\R^n$ is said to be \emph{$k$-regulous}, where $k$ is
a nonnegative integer, if it is of class~$\C^k$ and there exist two
polynomial functions $p$, $q$ on $\R^n$ such that $Z(q) \neq \R^n$ and
$f = p/q$ on $\R^n \setminus Z(q)$; \mbox{a~$0$-regulous} function is called
\emph{regulous}. Such functions, which often appear under different
names and in more general contexts, have several remarkable properties
and applications, cf. \citeleft \citen{bib1}\citepunct \citen{bib4}\citepunct
\citen{bib5}\citepunct \citen{bib8}\citedash\citen{bib22}\citepunct
\citen{bib25}\citepunct \citen{bib27}\citeright. In particular, the
authors of \cite{bib4} obtained the following: a variant of the
classical Nullstellensatz for the ring $\RC^k(\R^n)$ of $k$-regulous
functions on $\R^n$, a description of the zero locus~$Z(F)$ of an
arbitrary collection $F \subseteq \RC^k(\R^n)$ in terms of Zariski
(algebraically) constructible sets, and counterparts of Cartan's
theorems~A and~B for quasi-coherent $k$-regulous sheaves.

For the sake of clarity, let us recall that a function $f$ on $\R^n$
which is regulous and of class $\Cinfty$ is actually regular, that is,
$f=p/q$ for some polynomial functions $p$, $q$ on~$\R^n$ with $Z(q) =
\varnothing$, cf. \cite[Proposition~2.1]{bib10}. Therefore one gains no
new insight by considering such functions.

In the present paper we introduce \emph{Nash $k$-regulous functions} and
show that most results of \cite{bib4} can be carried over to this new
setting. Along the way we establish connections with arc-analytic
functions and constructible categories investigated in \cite{bib1a,
bib23, bib24, bib26}, and also point out at a potential application to
the problem of continuous solutions of linear equations \cite{bib3}.

Recall that a \emph{Nash manifold}~$X$ is an analytic
submanifold of~$\R^n$, for some $n$, which is also a semialgebraic set.
A \emph{Nash function} on $X$ is an analytic function with semialgebraic
graph. The ring $\NC(X)$ of Nash functions on $X$ is Noetherian
\cite[Theorem~8.7.18]{bib2}.

%Definition 1.1
\begin{definition}\label{def-1-1}
A function $f$ on $X$ is said to be \emph{Nash $k$-regulous}, where $k$
is a nonnegative integer, if it is of class $\C^k$ and there exist two
Nash functions $\varphi$, $\psi$ on $X$ such that the set $Z(\psi)$ is
nowhere dense in $X$ and $f=\varphi/\psi$ on $X \setminus Z(\psi)$; a
Nash $0$-regulous function is called \emph{Nash regulous}.
\end{definition}

The set $\NC^k(X)$ of all Nash $k$-regulous functions on $X$ forms a
ring. Obviously, $\RC^k(\R^n)$ is a subring of $\NC^k(\R^n)$.

%Example 1.2
\begin{example}\label{ex-1-2}
The function $f \colon \R^2 \to \R$ defined by
\begin{equation*}
f(x,y) = \frac{x^{3+k}}{x^2+y^2} \ \textrm{for} \ (x,y) \neq (0,0)\
\textrm{and} \ f(0,0)=0
\end{equation*}
is $k$-regulous but it is not $(k+1)$-regulous. In turn, the function $g
\colon \R^2 \to \R$, $g(x,y) \coloneqq f(x,y) \sqrt{1+x^2}$ is Nash
$k$-regulous but it is not $k$-regulous.
\end{example}

A more interesting example is the following.

%Example 1.3
\begin{example}\label{ex-1-3}
Consider the function $h \colon \R^3 \to \R$ defined by
\begin{equation*}
h(x,y,z) =
\begin{cases}
\frac{x^3}{x^2 + (1+z^2)^{1/3}xy + (1+z^2)^{2/3}y^2} & \textrm{if} \
x^2+y^2 \neq 0\\
0 & \textrm{otherwise.}
\end{cases}
\end{equation*}
Since $x^2+(1+z^2)^{1/3}xy + (1+z^2)^{2/3}y^2 \geq \frac{1}{2}
(x^2+y^2)$, we conclude that $h$ is a Nash regulous function.
\end{example}

It readily follows from \cite[Proposition~2.2.2]{bib2} that any Nash
regulous function is semialgebraic. Furthermore, by
\cite[Proposition~8.1.8]{bib2}, any function on $X$ which is Nash
regulous and of class $\Cinfty$ is actually a Nash function. Evidently,
if $\dim X =1$, then any Nash regulous function on $X$ is a Nash
function.

Our main results on geometric and algebraic properties of Nash
$k$-regulous functions are Theorem~\ref{th-3-5} (characterization of
$Z(F)$ for $F \subseteq \NC^k(X)$) and Theorem~\ref{th-3-6}
(Nullstellensatz for $\NC^k(X)$) in Section~\ref{sec-3}. They depend on
some constructions that involve Nash sets.

Recall that a subset $V \subseteq X$ is called a \emph{Nash subset} if
it can be written in the form $V = Z(F)$ for some collection $F$ of Nash
functions on $X$. Since the ring $\NC(X)$ is Noetherian, one can find
functions $f_1, \ldots, f_r$ in $F$ such that $V=Z(f)$ with $f =f_1^2 +
\cdots + f_r^2$. In particular, each Nash subset of $X$ is
semialgebraic.

%Definition 1.4
\begin{definition}\label{def-1-2}
A subset $E$ of $X$ is said to be \emph{Nash constructible} if it
belongs to the Boolean algebra generated by the Nash subsets of $X$; or
equivalently if $E$ is a finite union of sets of the form $V \setminus
W$, where $V$ and $W$ are Nash subsets of $X$.
\end{definition}

Basic properties of Nash constructible sets are established in
Section~\ref{sec-2} and used in the remainder of the paper.

Unless explicitly stated otherwise, we consider $X$ endowed with the
\emph{Euclidean topology} induced by the standard metric on $\R$.
However, the \emph{Nash topology} and the \emph{Nash constructible
topology}, both defined in Section~\ref{sec-2}, are also useful. In
Section~\ref{sec-4} we introduce a locally ringed space $(X, \NC_X^k)$,
with underlying Nash constructible topology, and study sheaves of
$\NC_X^k$-modules on $X$. The main results are Theorem~\ref{th-4-4}
(Cartan's theorem~A) and Theorem~\ref{th-4-5} (Cartan's theorem~B).
Essentially, the ringed space $(X, \NC_X^k)$ is equivalent to the affine
scheme $\Spec(\NC^k(X))$. We will see that the scheme $\Spec(\NC^k(X))$
is rather unusual for every open subset of it is affine. Let us recall
that Cartan's theorems~A and~B fail for coherent Nash sheaves, cf.
\cite{bib7}.

For background material on real algebraic geometry we refer the reader
to \cite{bib2}. The dimension of a semialgebraic set~$A$ in~$\R^n$,
written $\dim A$, is defined as the maximum dimension of a Nash
submanifold of $\R^n$ contained in $A$. Recall that $\dim A$ is equal to
the dimension of the Zariski closure of $A$ in $\R^n$. If $V$ is a Nash
subset of $X$, then $\dim V$ is, by definition, the dimension of $V$
regarded as a semialgebraic set.

%Section 2
\section{Nash constructible sets}\label{sec-2}

There are elementary concepts of irreducible subsets and decomposition
into irreducible components in an arbitrary topological space $S$. They
are particularly useful if $S$ is a Noetherian space, that is, if every
descending chain of closed subsets of $S$ is stationary. For these
notions \cite[pp.~13--15]{bib6} can be consulted.

Let $X$ be a Nash manifold. The collection of all Nash subsets of $X$ is
the family of closed subsets for some topology on $X$, called the
\emph{Nash topology}. Since the ring $\NC(X)$ of Nash functions on $X$ is
Noetherian, the Nash topology is Noetherian. It readily follows that if
$Z' \subseteq Z$ are Nash subsets of~$X$ with $Z' \neq Z$ and $Z$ Nash
irreducible, then $\dim Z' < \dim Z$ (recall our convention on
dimension).

Let $V \subseteq X$ be an irreducible Nash subset of dimension $d$. A
point $x \in V$ is said to be \emph{regular in dimension d} if there
exists an open (in the Euclidean topology) semialgebraic neighborhood $U
\subseteq X$ of $x$ such that $V \cap U$ is a Nash submanifold of $U$ of
dimension $d$. The set $\Reg_d(V)$ of all regular points of $V$ in
dimension $d$ is a semialgebraic subset for which
\begin{equation*}
\dim(V \setminus \Reg_d(V)) < d.
\end{equation*}
Denoting by $\Vsing$ the closure of $V \setminus \Reg_d(V)$ in the Nash
topology on $X$, we get
\begin{equation*}
\dim(V \setminus \Reg_d(V)) = \dim \Vsing.
\end{equation*}

Now let $W \subseteq X$ be a Nash subset and let $W_1, \ldots, W_r$ be
its Nash irreducible components. We set
\begin{equation*}
\Wsing \coloneqq \bigcup_i \Wsing_i \cup \bigcup_{j \neq k} (W_j \cap
W_k) \quad \textrm{and} \quad \Wns \coloneqq W \setminus \Wsing.
\end{equation*}
Consequently, $\Wsing$ is a Nash subset of $X$ with 
\begin{equation*}
\dim \Wsing < \dim W.
\end{equation*}

%Proposition 2.1
\begin{proposition}\label{prop-2-1}
Let $X$ be a Nash manifold, $E \subseteq X$ a closed (in the Euclidean
topology) Nash constructible subset, and $W$ the Nash closure of $E$ in
$X$. Then $\Wns \subseteq E$.
\end{proposition}

\begin{proof}
The set $E$ can be written as a finite union
\begin{equation*}
E = \bigcup_{i \in I} (Z_i \setminus Z'_i),
\end{equation*}
where $Z'_i \subseteq Z_i$ are Nash subsets of X with $Z'_i \neq Z_i$
and $Z_i$ Nash irreducible for all $i \in I$. Note that $Z_i \subseteq
W$ because $Z_i \setminus Z_i' \subseteq W$ and the Nash closure of $Z_i
\setminus Z_i'$ in $X$ is equal to $Z_i$. Let
$W_1, \ldots, W_r$ be the Nash irreducible components of $W$. Since
$Z_i$ is Nash irreducible, we have $Z_i \subseteq W_j$ for some~$j$.
Setting $I_j \coloneqq \{ i \in I \colon Z_i \subseteq W_j \}$, we get
\begin{equation*}
W_j = \bigcup_{i \in I_j} Z_i
\end{equation*}
for all $j = 1, \ldots, r$. Furthermore, if
\begin{equation*}
W'_j \coloneqq \bigcup_{i \in I_j} Z'_i,
\end{equation*}
then $\dim W'_j < \dim W_j$ and $W_j \setminus W'_j \subseteq E$. In
particular, $\Wns_j \setminus W'_j \subseteq E$ for every~$j$. Since
$\Wns_j \setminus W'_j$ is Euclidean dense in $\Wns_j$, we get $\Wns_j
\subseteq E$ ($E$ is closed in the Euclidean topology). Hence
\begin{equation*}
\Wns \subseteq \Wns_1 \cup \ldots \cup \Wns_r \subseteq E,
\end{equation*}
as required.
\end{proof}

%Proposition 2.2
\begin{proposition}\label{prop-2-2}
Let $X$ be a Nash manifold and let
\begin{equation*}
E_1 \supseteq E_2 \supseteq E_3 \supseteq \cdots
\end{equation*}
be a chain of closed (in the Euclidean topology) Nash constructible
subsets of $X$. Then the chain is stationary, that is, there exists a
positive integer $m$ such that $E_m = E_i$ for all $i \geq m$.
\end{proposition}

\begin{proof}
Let $W_i$ be the Nash closure of $E_i$ in $X$. We use induction on $\dim
W_1$. The case $\dim W_1 = 0$ is obvious since then $W_1$ is a finite
set.

Suppose that $\dim W_1 > 0$. The chain
\begin{equation*}
W_1 \supseteq W_2 \supseteq W_3 \supseteq \cdots
\end{equation*}
of Nash subsets of $X$ is stationary, hence there exists a positive
integer $l$ such that $W_l = W_i$ for all $i \geq l$. Set $W \coloneqq
W_l$. By Proposition~\ref{prop-2-1}, we get a chain
\begin{equation*}
E_l \setminus \Wns \supseteq E_{l+1} \setminus \Wns \supseteq E_{l+2}
\setminus \Wns \supseteq \cdots
\end{equation*}
of closed (in the Euclidean topology) Nash constructible subsets of $X$.
This chain is stationary by the induction hypothesis since $E_l
\setminus \Wns \subseteq \Wsing$ and $\dim \Wsing < \dim W \leq \dim
W_1$. It follows that the chain
\begin{equation*}
E_l \supseteq E_{l+1} \supseteq E_{l+2} \supseteq \cdots
\end{equation*}
is stationary, which completes the proof.
\end{proof}

In view of Proposition~\ref{prop-2-2}, the collection of all closed (in
the Euclidean topology) Nash constructible subsets of $X$ forms the
family of closed sets for a Noetherian topology on~$X$, called the
\emph{Nash constructible topology}. The reader may compare this with the
constructions of topologies in \cite{bib4, bib23, bib24, bib26}.

%Section 3
\section{Nash \texorpdfstring{$k$}{k}-regulous functions}\label{sec-3}

Recall that a function $f$ on an analytic manifold $M$ is said to be
\emph{arc-analytic} if for every analytic arc $\gamma \colon (-1,1) \to
M$ the composite $f \circ \gamma$ is an analytic function, cf.
\cite{bib23}.

%Proposition 3.1
\begin{proposition}\label{prop-3-1}
Let $X$ be a Nash manifold and let $f$ be a Nash regulous function on
$X$. Then $f$~is a semialgebraic arc-analytic function, and its zero
locus $Z(f)$ is a Nash constructible set.
\end{proposition}

\begin{proof}
As we already observed in Section~\ref{sec-1}, $f$ is a semialgebraic
function.

For the proof of the other assertions we may assume that $X$ is
connected (in the Euclidean topology). Let $\varphi$, $\psi$ be Nash
functions on $X$ such that $Z(\psi) \neq X$ and $f = \varphi/\psi$ on $X
\setminus Z(\psi)$. By the Artin--Mazur theorem
\cite[Theorem~8.4.4]{bib2}, there exist a nonsingular irreducible
algebraic set $V \subseteq \R^m$, an open semialgebraic subset $X'
\subseteq V$, a Nash diffeomorphism ${\sigma \colon X \to X'}$, and two
polynomial functions $p$, $q$ on $\R^m$ such that
\begin{equation*}
p(\sigma(x)) = \varphi(x) \quad \textrm{and} \quad q(\sigma(x)) =
\psi(x) \quad \textrm{for all} \ x \in X.
\end{equation*}
Evidently, the function $g \coloneqq f \circ \sigma^{-1} \colon X' \to
\R$ is continuous and $g = p/q$ on $X' \cap (V \setminus Z(q))$. In view
of \cite[Proposition~4.2, Theorem~1.12]{bib8}, $g$ is an arc-analytic
function. Furthermore, according to \cite[Propositions~4.2
and~3.5]{bib8}, $Z(g)$ is a Nash constructible subset of $X'$. The proof
is complete since $\sigma$ is a Nash diffeomorphism.
\end{proof}

We note the following.

%Corollary 3.2
\begin{corollary}\label{cor-3-2}
Let $X$ be a Nash manifold and let $F$ be a collection of Nash regulous
functions on~$X$. Then the zero locus $Z(F)$ is a closed subset of $X$
in the Nash constructible topology.
\end{corollary}

\begin{proof}
This assertion holds since
\begin{equation*}
Z(F) = \bigcap_{f \in F}Z(f)
\end{equation*}
and, by Proposition~\ref{prop-3-1}, each set $Z(f)$ is closed in the
Nash constructible topology.
\end{proof}

It is convenient to generalize Definition~\ref{def-1-1} as follows.

%Definition 3.3
\begin{definition}\label{def-3-3}
Let $X$ be a Nash manifold and let $U \subseteq X$ be an open subset (in
the Euclidean topology). A function $f$ on $U$ is said to be \emph{Nash
$(X,k)$-regulous}, where $k$ is a nonnegative integer, if it is of class
$\C^k$ and there exist two Nash functions $\varphi$, $\psi$ on $X$ such
that the set $Z(\psi)$ is nowhere dense in $X$ and $f=\varphi/\psi$ on
$U \setminus Z(\psi)$.
\end{definition}

The set $\NC_X^k(U)$ of all Nash $(X,k)$-regulous functions on $U$ forms
a ring. It depends on the triple $(X, U, k)$ and not on the pair $(U,k)$
alone. If $U' \subseteq U$ is an open subset, then there is a
well-defined homomorphism
\begin{equation*}
\NC_X^k(U) \to \NC_X^k(U')
\end{equation*}
determined by the restriction of functions. Obviously, $\NC_X^k(X) =
\NC^k(X)$.

The following is a Nash $k$-regulous version of \cite[Lemma~5.1]{bib4}.

%Proposition 3.4
\begin{proposition}\label{prop-3-4}
Let $X$ be a Nash manifold, $k$ a nonnegative integer, $f \colon X \to
\R$ a Nash $k$-regulous function, and $g \colon X \setminus Z(f) \to \R$
a Nash $(X,k)$-regulous function. Then, for each integer~$N$ large
enough, the function $h \colon X \to \R$ defined by
\begin{equation*}
h = f^Ng \quad \textrm{on} \ X \setminus Z(f) \quad \textrm{and} \quad h=0
\quad \textrm{on} \ Z(f)
\end{equation*}
is Nash $k$-regulous.
\end{proposition}

\begin{proof}
The function $f$ is semialgebraic. Similarly, in view of
\cite[Proposition~2.2.2]{bib2}, $g$ is also semialgebraic. Consequently,
according to \cite[Proposition~2.6.4]{bib2}, $h$ is continuous if $N$ is
large enough. Note that then $h$ is actually a Nash regulous function.
Thus it suffices to prove that for each integer~$s$ large enough the
product $f^s g$, defined on $X \setminus Z(f)$, extends by zero through
the zero locus $Z(f)$ to a~$\C^k$ function on $X$. Instead of computing
partial derivatives in some local coordinates, it seems more convenient
to make use of globally defined vector fields on $X$.

Let $\eta$ be a Nash vector field on $X$ and let $u$ be a function of
class $\C^p$, where $p \geq 1$, defined on an open subset $U \subseteq
X$. By applying $\eta$ to $u$, we get a function $\eta u = \eta(u)$ of
class $\C^{p-1}$ on~$U$. If $u$~is a semialgebraic function (hence $U$
is a semialgebraic set), then so is $\eta u$. Given an integer $q$, with
$0 \leq q \leq p$, we define $\eta^q u$ recursively:
\begin{equation*}
\eta^q u = u \quad \textrm{for} \ q=0 \quad \textrm{and} \quad \eta^q u =
\eta(\eta^{q-1} u) \quad \textrm{for} \ q \geq 1.
\end{equation*}

Now let $\xi = (\xi_1, \ldots, \xi_n)$ be an $n$-tuple of Nash vector
fields that generate the tangent bundle to~$X$. For any $n$-tuple
$\alpha = (\alpha_1, \ldots, \alpha_n)$ of nonnegative integers, with
$|\alpha| \coloneqq \alpha_1 + \cdots + \alpha_n \leq k$, the function
\begin{equation*}
\xi^{\alpha}(g) \coloneqq \xi_1^{\alpha_1} \cdots \xi_n^{\alpha_n} g
\end{equation*}
is a continuous semialgebraic function on $X \setminus Z(f)$. By
\cite[Proposition~2.6.4]{bib2}, we can choose a positive integer $r$ so
that the functions
\begin{equation*}
f^r \xi^{\alpha}(g) \quad \textrm{for all} \ \alpha \ \textrm{with} \
|\alpha| \leq k
\end{equation*}
extend by zero to continuous functions on $X$. Hence, by the Leibniz
rule, $f^{r+k}g$ extends by zero to a function of class $\C^k$ on $X$,
which is $k$-flat on $Z(f)$.
\end{proof}

Proposition~\ref{prop-3-4} plays an essential role in the proof of the
following.

%Theorem 3.5
\begin{theorem}\label{th-3-5}
Let $X$ be a Nash manifold, $E \subseteq X$ some subset, and $k$ a
nonnegative integer. Then the following conditions are equivalent:
\begin{conditions}
\item\label{th-3-5-a} $E$ is closed in the Nash constructible topology.

\item\label{th-3-5-b} $E = Z(f)$ for some Nash $k$-regulous function $f$
on $X$.

\item\label{th-3-5-c} $E = Z(F)$ for some collection $F$ of Nash
$k$-regulous functions on $X$.
\end{conditions}
Furthermore, if \ref{th-3-5-c} holds, then there exist functions $g_1,
\ldots, g_r$ in $F$ such that
%---not displayed in manuscript
\begin{equation*}
E = Z(g) \quad \textrm{with} \ g = g_1^2 + \cdots + g_r^2.
\end{equation*}
\end{theorem}

\begin{proof}
Suppose that condition \ref{th-3-5-a} is satisfied, and let $W \subseteq
X$ be the closure of $E$ in the Nash topology. We use induction on $\dim
W$ to prove that \ref{th-3-5-b} holds. The case $\dim W = 0$ is obvious
since then $W$ is a finite set.

Assume that $\dim W > 0$. By Proposition~\ref{prop-2-1}, $\Wns \subseteq
E$ and hence
\begin{equation*}
W = E \cup Z, \quad \textrm{where} \ Z = \Wsing = W \setminus \Wns.
\end{equation*}
Let $\varphi$ and $\psi$ be Nash functions on $X$ with $Z(\varphi) = W$
and $Z(\psi) = Z$. Consider the functions
\begin{gather*}
\frac{1}{\varphi^2} \colon X \setminus W = (X \setminus E) \setminus (Z
\setminus (E \cap Z)) \to \R,\\
\psi|_{X \setminus E} \colon X \setminus E \to \R.
\end{gather*}
Since $Z(\psi|_{X \setminus E}) = Z \setminus (E \cap Z)$, it follows
from Proposition~\ref{prop-3-4} that for a sufficiently large
integer~$N$ the function $\alpha \colon X \setminus E \to \R$ defined by
\begin{equation*}
\alpha = \frac{\psi^{2N}}{\varphi^2} \quad \textrm{on} \ X \setminus W
\quad \textrm{and} \quad \alpha=0 \quad \textrm{on} \ Z \setminus (E \cap Z)
\end{equation*}
is Nash $k$-regulous. Clearly, $\alpha$ is actually Nash
$(X,k)$-regulous.

We claim that the function $\beta \colon X \setminus (E \cap Z) \to \R$
defined by
\begin{equation*}
\beta = \frac{\varphi^2}{\varphi^2 + \psi^{2N}} \quad \textrm{on} \ X
\setminus Z \quad \textrm{and} \quad \beta=1 \quad \textrm{on} \ Z \setminus (E \cap
Z)
\end{equation*}
is Nash $(X,k)$-regulous. Indeed, since
\begin{equation*}
X \setminus (E \cap Z) = (X \setminus E) \cup (X \setminus Z) \quad
\textrm{and} \quad (X \setminus E) \cap (X \setminus Z) = X \setminus W,
\end{equation*}
it suffices to show that $\beta$ is of class $\C^k$ on $X \setminus E$.
However, on $X \setminus E$ we have $\beta = \frac{1}{1 + \alpha}$ which
proves the claim.

By construction,
\begin{equation*}
Z(\beta) = W \setminus Z = E \setminus (E \cap Z).
\end{equation*}
Since $\dim Z < \dim W$, it follows from the induction hypothesis that
there exists a Nash $k$-regulous function $\gamma \colon X \to \R$ with
$Z(\gamma) = E \cap Z$. In view of Proposition~\ref{prop-3-4} once
again, if $M$ is a sufficiently large integer, then the function $f
\colon X \to \R$ defined by
\begin{equation*}
f = \gamma^M \beta \quad \textrm{on} \ X \setminus (E \cap Z) \quad \textrm{and}
\quad f=0 \quad \textrm{on} \ E \cap Z
\end{equation*}
is Nash $k$-regulous. Evidently, $Z(f)=E$ hence \ref{th-3-5-b} holds.

It is clear that \ref{th-3-5-b} implies \ref{th-3-5-c}, while in turn
\ref{th-3-5-c} implies \ref{th-3-5-a} by Corollary~\ref{cor-3-2}. In
conclusion, conditions \ref{th-3-5-a}, \ref{th-3-5-b}, \ref{th-3-5-c}
are equivalent.

For the last assertion, suppose that condition \ref{th-3-5-c} holds. It
suffices to prove that $E = Z(g_1, \ldots, g_r)$ for some functions
$g_1, \ldots, g_r$ in $F$. If this were not the case we could choose
functions $g_1, g_2, \ldots$ in~$F$ for which
\begin{equation*}
Z(g_1) \supsetneqq Z(g_1, g_2) \supsetneqq \cdots
\end{equation*}
However, such a chain cannot exist, the Nash constructible topology on
$X$ being Noetherian.
\end{proof}

There are analogous results to Theorem~\ref{th-3-5} for $k$-regulous
functions \cite[Th\'eor\`eme~5.21 et 6.4]{bib4} and for arc-analytic
functions \cite[Theorem~1.1]{bib1a}.

In the remainder of this section we obtain counterparts of suitable
results for $k$-regulous functions~\cite{bib4}.

We investigate ideals in the ring of Nash $k$-regulous functions on
a Nash manifold $X$. Given a subset $E \subseteq X$, denote by
$J_{\NC^k(X)}(E)$ the ideal of $\NC^k(X)$ comprised of all functions
vanishing on~$E$,
\begin{equation*}
J_{\NC^k}(E) \coloneqq \{ f \in \NC^k(X) \colon f(x) = 0 \quad \textrm{for
all} \ x \in E \}.
\end{equation*}
As usual, the radical of an ideal $I$ of $\NC^k(X)$ will be denoted by
$\Rad(I)$.

We have the following variant of the Nullstellensatz for the ring
$\NC^k(X)$.

%Theorem 3.6
\begin{theorem}\label{th-3-6}
Let $X$ be a Nash manifold and let $I$ be an ideal of the ring
$\NC^k(X)$, where $k$ is a nonnegative integer. Then
\begin{equation*}
J_{\NC^k}(Z(I)) = \Rad(I).
\end{equation*}
\end{theorem}

\begin{proof}
It is clear that $\Rad(I) \subseteq J_{\NC^k} (Z(I))$.

For the converse
inclusion, we first pick a function $g \in I$ with $Z(g) = Z(I)$; this
is possible by Theorem~\ref{th-3-5}. Consider a function $f \in
J_{\NC^k}(Z(I))$. Note that $X \setminus Z(f) \subseteq X \setminus
Z(g)$, and $1/g$ is a Nash $(X,k)$-regulous function on $X \setminus
Z(g)$. Hence, by Proposition~\ref{prop-3-4}, there exists a positive
integer~$N$ such that the function $h$ on $X$ defined by
\begin{equation*}
h = \frac{f^N}{g} \quad \textrm{on} \ X \setminus Z(g) \quad
\textrm{and} \quad h=0 \quad \textrm{on} \ Z(g)
\end{equation*}
is Nash $k$-regulous. Since $f^N = gh$, we get $f \in \Rad(I)$, as
required.
\end{proof}

As a straightforward consequence of Theorems \ref{th-3-5} and
\ref{th-3-6} we get the following.

%Corollary 3.7
\begin{corollary}\label{cor-3-7}
Let $X$ be a Nash manifold and let $k$ be a nonnegative integer. Then the
assignment
\begin{equation*}
X \supseteq E \mapsto J_{\NC^k}(E) \subseteq \NC^k(X)
\end{equation*}
gives rise to one-to-one correspondences:
\begin{iconditions}
\item\label{cor-3-7-i} between the closed subsets of $X$ in the Nash
constructible topology and the radical ideals of $\NC^k(X)$;

\item\label{cor-3-7-ii} between the irreducible closed subsets of $X$ in
the Nash constructible topology and the prime ideals of $\NC^k(X)$;

\item\label{cor-3-7-iii} between the points of $X$ and the maximal
ideals of $\NC^k(X)$. \qed
\end{iconditions}
\end{corollary}

For a function $f$ in $\NC^k(X)$, we let $\Rad(f)$ denote the radical of
the principal ideal generated by~$f$.

The following is a direct consequence of Theorem~\ref{th-3-6}.

%Corollary 3.8
\begin{corollary}\label{cor-3-8}
Let $X$ be a Nash manifold and let $f$ and $g$ be functions in the ring
$\NC^k(X)$, where $k$~is a nonnegative integer. Then $\Rad(f) = \Rad(g)$
if and only if $Z(f) = Z(g)$. \qed 
\end{corollary}

The radical of an arbitrary ideal can be described as follows.

%Corollary 3.9
\begin{corollary}\label{cor-3-9}
Let $X$ be a Nash manifold and let $I$ be an ideal of the ring
$\NC^k(X)$, where $k$ is a nonnegative integer. Then
\begin{equation*}
\Rad(I)=\Rad(f)
\end{equation*}
for some $f \in I$. Furthermore, the equality of radicals holds if and
only if $Z(I) = Z(f)$.
\end{corollary}

\begin{proof}
By Theorem~\ref{th-3-5}, $Z(I) = Z(f)$ for some $f \in I$. The proof is
complete in view of Theorem~\ref{th-3-6}.
\end{proof}

%Proposition~\ref{prop-3-4} allows us to give a geometric description of
%localization, which will be useful in Section~\ref{sec-4}.
As usual, for $f$ in $\NC^k(X)$, we let $\NC^k(X)_f$ denote the
localization of the ring $\NC^k(X)$ with respect to the multiplicatively
closed subset $\{1\} \cup \{f^m \colon m=1,2,\ldots\}$. In particular,
$\NC^k(X)_f$ is the zero ring if $f$ is identically equal to $0$. By
convention, $\NC^k_X(\varnothing)$ is also the zero ring.
Proposition~\ref{prop-3-4} allows us to give a geometric description of
$\NC^k(X)_f$, which will be useful in Section~\ref{sec-4}.

%Proposition 3.10
\begin{proposition}\label{prop-3-10}
Let $X$ be a Nash manifold, $k$ a nonnegative integer, and $f$ a Nash
$k$-regulous function on $X$. Then the restriction homomorphism
$\NC^k(X) \to \NC_X^k(X \setminus Z(f))$ induces an isomorphism between
the localization $\NC^k(X)_f$ and $\NC_X^k(X \setminus Z(f))$.
\end{proposition}

\begin{proof}
By Proposition~\ref{prop-3-4}, the induced homomorphism $\NC^k(X)_f \to
\NC_X^k(X \setminus Z(f))$ is surjective.

If an element $g/f^m \in \NC^k(X)_f$ is sent to $0 \in \NC_X^k(X
\setminus Z(f))$, then $fg=0$ in $\NC^k(X)$. Hence $g/f^m = 0$, which
means that the induced homomorphism is injective.
\end{proof}

If $\dim X = 1$, then $\NC^k(X) = \NC(X)$ is a Noetherian ring. The case
$\dim X \geq 2$ is entirely different.

%Remark 3.11
\begin{remark}\label{rem-3-11}
The ring $\NC^k(\R^n)$ is not Noetherian if $n \geq 2$ and $k \geq 0$.
Indeed, by Proposition~\ref{prop-3-1}, Nash regulous functions are
arc-analytic, and hence the argument in \cite[Exemple~6.11]{bib6} for
the ring of arc-analytic functions can be easily adapted to the case
under consideration (cf. also \cite[Proposition~4.16]{bib4}).

It would not be much harder to prove that the ring $\NC^k(X)$ is not
Noetherian for any Nash manifold $X$ of dimension at least~$2$. We leave
details to the interested reader.
\end{remark}

Our interest in Nash regulous functions originated from an attempt, not
conclusive yet, to strengthen a result of \cite{bib3}.

%Problem 3.12
\begin{problem}\label{prob-3-12}
Consider a linear equation
\begin{equation*}
f_1 y_1 + \cdots + f_r y_r = g,
\end{equation*}
where $g$ and the $f_i$ are polynomial (or regular) functions on $\R^n$.
Assume that it admits a solution where the $y_i$ are continuous
functions on $\R^n$. Then, according to \cite[Section~2]{bib3}, it has
also a continuous semialgebraic solution. One could hope to prove that
it has a regulous solution. This is indeed the case for $n=2$
\cite[Corollary~1.7]{bib20}, but fails for any $n \geq 3$
\cite[Example~6]{bib9}.

It would be interesting to decide whether or not the equation has a Nash
regulous solution.
\end{problem}

%Section 4
\section{Nash \texorpdfstring{$k$}{k}-regulous sheaves}\label{sec-4}

In order to avoid awkward repetitions, we begin by fixing some notation.

%Notation 4.1
\begin{notation}\label{not-4-1}
Throughout this section, $X$ stands for a Nash manifold and $k$ stands
for a nonnegative integer. We will consider $X$ endowed with the
\emph{Nash constructible topology} (Section~\ref{sec-2}). Thus, a subset
$E \subseteq X$ is closed if and only if $E = Z(f)$ for some $f \in
\NC^k(X)$ (Theorem~\ref{th-3-5}). We set $X(f) \coloneqq X \setminus
Z(f)$.
\end{notation}

For every open subset $U$ of $X$, the ring $\NC_X^k(U)$ of Nash
$(X,k)$-regulous functions is defined (Definition~\ref{def-3-3}). It
readily follows that the assignment
\begin{equation*}
\NC_X^k \colon U \mapsto \NC_X^k(U)
\end{equation*}
is a sheaf of rings on $X$, and $(X, \NC_X^k)$ is a locally ringed
space.

There is a close connection between the ringed space $(X, \NC_X^k)$ and
the affine scheme $\Spec(\NC^k(X))$. We first describe relationships
between the underlying topological spaces.

By definition, a subset $V \subseteq \Spec(\NC^k(X))$ is closed if and
only if $V = V(F)$ for some collection $F \subseteq \NC^k(X)$, where
\begin{equation*}
V(F) \coloneqq \{ \pf \in \Spec(\NC^k(X)) \colon F \subseteq \pf \}.
\end{equation*}
If $I$ is the ideal generated by $F$, then $V(F) = V(\Rad(I))$. Hence,
according to Corollary~\ref{cor-3-9}, $V(F) = V(f)$ for some $f \in
I$; here $V(f) \coloneqq V(\{f\})$. Consequently, each open subset of
$\Spec(\NC^k(X))$ is of the form
\begin{equation*}
D(f) \coloneqq \Spec(\NC^k(X)) \setminus V(f)
\end{equation*}
for some $f \in \NC^k(X)$. In particular, each open subset of
$\Spec(\NC^k(X))$ is affine.

Define a map
\begin{equation*}
\iota \colon X \to \Spec(\NC^k(X))
\end{equation*}
by $\iota(x) = \mf_x$ for all $x \in X$, where
\begin{equation*}
\mf_x \coloneqq \{f \in \NC^k(X) \colon f(x)=0 \}
\end{equation*}
($\mf_x = J_{\NC^k}(\{x\})$ with notation as in Section~\ref{sec-3}).

%Proposition 4.2
\begin{proposition}\label{prop-4-2}
The map $\iota$ is a topological embedding of $X$ onto the subspace
$\Max(\NC^k(X))$ of $\Spec(\NC^k(X))$ comprised of the maximal ideals of
$\NC^k(X)$. Furthermore, the following conditions hold:
\begin{nconditions}
\item\label{prop-4-2-1} $Z(f) = \iota^{-1}(V(f))$ for every $f \in
\NC^k(X)$.

\item\label{prop-4-2-2} For each closed subset $Z \subseteq X$ there
exists a unique closed subset $\tilde Z \subseteq \Spec(\NC^k(X))$ such
that $Z = \iota^{-1}(\tilde Z)$.

\item\label{prop-4-2-3} $X(f) = \iota^{-1}(D(f))$ for every $f \in
\NC^k(X)$.

\item\label{prop-4-2-4} For each open subset $U \subseteq X$ there
exists a unique open subset $\tilde U \subseteq \Spec(\NC^k(X))$ such
that $U = \iota^{-1}(\tilde U)$.
\end{nconditions}
\end{proposition}

\begin{proof}
By Corollary~\ref{cor-3-7}, $\iota$ includes a bijection of $X$ onto
$\Max(\NC^k(X))$. It follows immediately that \ref{prop-4-2-1} holds.

Suppose that $Z \subseteq X$ and $V \subseteq \Spec(\NC^k(X))$ are
closed subsets with $Z = \iota^{-1}(V)$. Then $Z = Z(f)$ and $V = V(g)$
for some $f, g \in \NC^k(X)$. Hence $Z(f) = Z(g)$ and, by
Corollary~\ref{cor-3-8}, $\Rad(f) = \Rad(g)$. Consequently, $V(f) = V(g)
= V$, which proves \ref{prop-4-2-2}.

Conditions \ref{prop-4-2-3} and \ref{prop-4-2-4} follow from conditions
\ref{prop-4-2-1} and \ref{prop-4-2-2}, respectively.

It is now clear that $\iota$ induces a homeomorphism between $X$ and
$\Max(\NC^k(X))$.
\end{proof}

It is worthwhile to record the following.

%Corollary 4.3
\begin{corollary}\label{cor-4-3}
The topological space $\Spec(\NC^k(X))$ is Noetherian.
\end{corollary}

It should be mentioned that Proposition~\ref{prop-4-2} and
Corollary~\ref{cor-4-3} are Nash $k$-regulous versions of
\cite[Th\'eor\`eme~5.29, Corollaires~5.30, 5.31, 5.32]{bib4}.

\begin{proof}
This assertion follows from Proposition~\ref{prop-4-2} since $X$ is a
Noetherian topological space.
\end{proof}

Next we study sheaves on the spaces under consideration. For any
function $f \in \NC^k(X)$, we identify the rings $\NC^k(X)_f$ and
$\NC_X^k(X(f))$ via the canonical isomorphism described in
Proposition~\ref{prop-3-10}. Denoting by $\tilde{\NC}_X^k$ the structure
sheaf on $\Spec(\NC^k(X))$, we get
\begin{equation*}
\iota_* \NC_X^k = \tilde{\NC}_X^k \quad \textrm{and} \quad \NC_X^k = \iota^{-1}
\tilde{\NC}_X^k.
\end{equation*}

In view of Proposition~\ref{prop-4-2}, the category of sheaves of
Abelian groups (resp.\ sheaves of $\NC_X^k$-modules) on $X$ is equivalent
to the category of sheaves of Abelian groups (resp.\ sheaves of
$\tilde{\NC}_X^k$-modules) on $\Spec(\NC^k(X))$. The equivalence is
effected by the direct image functor ${\FC \mapsto \iota_*\FC}$, whose
inverse is the inverse image functor $\GC \mapsto \iota^{-1}\GC$. Via
these equivalences, quasi-coherent sheaves of $\NC_X^k$-modules on $X$
correspond to quasi-coherent sheaves of $\tilde{\NC}_X^k$-modules on
$\Spec(\NC^k(X))$.

Henceforth, by a \emph{Nash $k$-regulous sheaf} on $X$ we mean a sheaf
of $\NC_X^k$-modules. Our goal is to establish basic properties of
quasi-coherent Nash $k$-regulous sheaves.

For the sake of clarity of the exposition, it is convenient to recall
Cartan's theorem~A for affine schemes (cf. \cite[Theorem~7.16,
Corollary~7.17]{bib6}). For a commutative ring~$R$, consider the affine
scheme $Y = \Spec(R)$ with structure sheaf $\OC_Y$. Any $R$-module $M$
determines a quasi-coherent sheaf~$\tilde M$ of $\OC_Y$-modules on $Y$.
The functor $M \mapsto \tilde M$ gives an equivalence of categories
between the category of $R$-modules and the category of quasi-coherent
sheaves of $\OC_Y$-modules on $Y$. Its inverse is the global section
functor $\GC \mapsto \GC(Y)$. In particular, every quasi-coherent sheaf
of $\OC_Y$-modules on $Y$ is generated by its global sections.

Returning to our main topic, for an $\NC^k(X)$-module $M$, we define a
presheaf $\tilde M_X$ of $\NC_X^k$-modules on $X$ by
\begin{equation*}
\tilde M_X(U) \coloneqq M \otimes_{\NC^k(X)} \NC_X^k(U)
\end{equation*}
for every open subset $U \subseteq X$. Since $U = X(f)$ for some $f \in
\NC^k(X)$, we get
\begin{equation*}
\tilde M_X(U) = \tilde M_X(X(f)) = M \otimes_{\NC^k(X)} \NC^k(X)_f =
\tilde M (D(f)),
\end{equation*}
where the last equality is the canonical identification. Hence $\tilde
M_X$ is actually a sheaf and
\begin{equation*}
\tilde M_X = \iota^{-1} \tilde M.
\end{equation*}

We get immediately the following variant of Cartan's theorem~A.

%Theorem 4.4
\begin{theorem}\label{th-4-4}
The functor $M \mapsto \tilde M_X$ gives an equivalence of categories
between the category of $\NC^k(X)$-modules and the category of
quasi-coherent Nash $k$-regulous sheaves on $X$. Its inverse is the
global section functor $\FC \mapsto \FC(X)$. In particular, every
quasi-coherent Nash $k$-regulous sheaf on~$X$ is generated by its
global sections. \qed
\end{theorem}

According to Cartan's theorem~B for affine schemes, if $\GC$ is a
quasi-coherent sheaf of $\OC_Y$-modules on $Y = \Spec(R)$, then $H^i(Y,
\GC) = 0$ for $i \geq 1$ (cf. \cite[Th\'eor\`eme~1.3.1]{bib2a}).

The equivalence of the categories of sheaves on $X$ and on
$\Spec(\NC^k(X))$ via the functors $\iota_*$ and $\iota^{-1}$ yield the
following variant of Cartan's theorem~B.

%Theorem 4.5
\begin{theorem}\label{th-4-5}
If $\FC$ is a quasi-coherent Nash $k$-regulous sheaf on $X$, then
\begin{equation*}
\pushQED{\qed}
H^i(X, \FC) = 0 \quad \textrm{for all} \ i \geq 1. \qedhere
\popQED
\end{equation*}
\end{theorem}

It should be mentioned that Theorems \ref{th-4-4} and \ref{th-4-5} are
analogous to the results on quasi-coherent $k$-regulous sheaves obtained
in \cite{bib4}.

%References
%\cleardoublepage
\phantomsection

\end{document}